\documentclass[10pt,a4paper]{article}
\usepackage[utf8]{inputenc}
\usepackage[T1]{fontenc}
\usepackage{amsmath}
\usepackage{amsthm}
\usepackage{amsfonts}
\usepackage{amssymb}
\usepackage{makeidx}
\usepackage{graphicx}
\usepackage{xspace}
\usepackage{algorithm}
\usepackage{algorithmic}
\usepackage{lmodern}
\usepackage{mathtools}
\usepackage[left=2cm,right=2cm,top=1cm,bottom=2cm]{geometry}
\usepackage{cleveref}
\parskip \baselineskip

\newcommand{\bc}{\begin{center}}
\newcommand{\ec}{\end{center}}
\newcommand{\be}{\begin{enumerate}}
\newcommand{\ee}{\end{enumerate}}
\newcommand{\beq}{\begin{equation}}
\newcommand{\eeq}{\end{equation}}
\newcommand{\bi}{\begin{itemize}}
\newcommand{\ei}{\end{itemize}}
\newcommand{\bd}{\begin{description}}
\newcommand{\ed}{\end{description}}
\newcommand{\ba}{\begin{array}}
\newcommand{\bea}{\begin{eqnarray*}}
\newcommand{\eea}{\end{eqnarray*}}
\newcommand{\ea}{\end{array}}
\newcommand{\bt}{\begin{tabular}}
\newcommand{\et}{\end{tabular}}

\newcommand{\bmi}{\begin{minipage}}
\newcommand{\emi}{\end{minipage}}

\newcommand{\M}{{\cal M}}

\newtheorem{thm}{Theorem}[section]

\newtheorem{lem}[subsection]{Lemma}
\newtheorem{cor}[subsection]{Corollary}

\newtheorem{pro}[subsection]{Proposition}
\newtheorem{exa}[subsection]{Example}
\newtheorem{defn}[subsection]{Definition}
\newtheorem{rem}[subsection]{Remark}

 \newtheorem{ques}[subsection]{Question}

\begin{document}

\bc {\bf\large On Linear Maps and Seed Sets of Beidleman Near-Vector Spaces }\\[3mm]
{\sc P. Djagba and A.L. Prins  }

\it\small
Department of Mathematics and Applied Mathematics,\\
Nelson Mandela University, South Africa\\

\rm e-mail: prudence@aims.ac.za, abrahamprinsie@yahoo.com 
\ec
 
\normalsize

\quotation{\small {\bf Abstract:} 
We studied linear mappings in Beidleman near-vector spaces and explored their matrix representations using $R$-bases of $R$-subgroups. Additionally, we developed algorithms for determining the seed number and seed sets of $R$-subgroups within finite-dimensional Beidleman near-vector spaces.
\normalsize

\textup{2010} \textit{MSC}: \textup{16Y30;12K05}

\section{Introduction}

Nearfields, first studied by Dickson \cite{dickson1905finite} in $1905$, found immediate applications in geometry. Despite their close resemblance to fields, the absence of one-sided distributive laws makes the study of nearfields challenging.

Nearfields, also known as skewfields or division rings, lack the distributive law on one side. Dickson's pioneering work in 1905 initiated their exploration, revealing connections to geometry and automata theory \cite{veblen1907, appl1980, automata}. Most finite nearfields are constructed by distorting multiplication in finite fields through Dickson's method, with seven exceptional examples 		\cite{zassenhauss1935}.  For a comprehensive overview, consult books by Pilz \cite{pilz2011near} and Meldrum \cite{meldrum1985near}, among others \cite{zassenhauss1935,ellerskarzel1964, wahling1987theorie,finitenearfield, djagba, djagbacenter}.  In $1966$, Beidleman introduced the concept of near-vector spaces over nearfields, employing nearring modules and the left distributive law \cite{beidleman1966near}. A different notion of near-vector spaces defined by Andr\'{e} in 1974, utilizing automorphisms, results in the right distributive law \cite{andre1974, Howell2008, dorfling2018decomposition, sanon2020}.

More recent contributions to the theory of Beidleman near-vector spaces were made by Djagba and Howell \cite{djagbathesis,djagbahowell18,djagbajan}. These contributions delve into subspaces and subgroups of near-vector spaces over nearfield notions like $R$-dimension, $R$-basis, seed set, and seed number of an $R$-subgroup were introduced. Due to the lack of distributivity, near-vector spaces exhibit more anomalous behavior compared to vector spaces over fields. An $R$-subgroup of a near-vector space is a subset closed under vector addition and vector-scalar multiplication. It can be generated by a set of vectors, with explicit procedures like 'Expanded Gaussian Elimination' \cite{djagbahowell18,djagbajan} characterizing $R$-subgroups generated by finite sets of vectors. This result implies that a near-vector space $R^m$ over a proper nearfield $R$ can be generated by fewer than $m$ vectors.

   This paper focuses on the earlier Beidleman definition. We continue to explore the subgroup structure of finite-dimensional Beidleman near-vector spaces, focusing on the canonical case of $R^m$. We derive matrix representations of linear and normal linear mappings between finite-dimensional Beidleman near-vector spaces. Unlike vector spaces, the set of linear mappings from near-vector spaces does not form a nearring. Finally, we present an explicit algorithm to determine the seed sets and seed number of $R$-subgroups.

\section{Preliminaries}
Let $R$ be a non-empty set.
\begin{defn}(\cite{meldrum1985near})
The triple $(R,+,\cdot)$ is a (left) nearring  if $(R,+)$ is a group,
 $(R,\cdot)$ is a semigroup, and  $a(b+c)= ab+ac$ for all $a,b,c \in R.$
\end{defn}

A nearfield is an algebraic structure similar to a skew-field, also known as a division ring. The key distinction is that it has only one of the two distributive laws.  
\begin{defn} (\cite{pilz2011near})
  Let $R$ be nearring. If $ \big ( R^*=R \setminus \{0 \}, \cdot \big )$ is a group then $(R,+, \cdot)$ is called nearfield.
\label{th:t8}
 \end{defn}
In this paper, we will utilize left nearfields and right nearring modules. Various mathematicians, including Dickson, Zassenhauss, Neumann, Karzel, and Zemmer, have independently demonstrated that the additive group of a nearfield is abelian.
 \begin{thm}(\cite{pilz2011near})
 The additive group of nearfield is abelian.
 \label{tt}
 \end{thm}
To construct finite Dickson nearfields, we require two concepts:

 \begin{defn} (\cite{pilz2011near})
A pair of numbers $(q,n) \in \mathbb{N}^2$ is called a Dickson pair if
$q$ is some power $p^l$ of a prime $p$, each prime divisor of $n$ divides $q-1$,  $q \equiv 3$ $ \text{mod } 4$ implies $4$ does not divide $n$.
\end{defn}

 \begin{defn}(\cite{pilz2011near})
Let $R$ be a nearfield and $\textit{Aut} (R,+,\cdot ) $ the set of all automorphisms of $N$. A map
\begin{align*}
\phi: \quad & R^* \to  \textit{Aut} (R,+,\cdot )  \\ 
& n \mapsto \phi_n
\end{align*}
is called a coupling map if for all $n,m \in R^*, \phi _n \circ \phi_m= \phi _{ \phi _n (m) \cdot n}.$
\end{defn}
Dickson's pioneering work in $1905$ led to the discovery of the first proper finite nearfield. He achieved this by distorting the multiplication operation of a finite field. For any pair of Dickson numbers $(q, n)$, there exist corresponding finite Dickson nearfields with an order of $q^n$. These nearfields are obtained by starting with the Galois field $GF(q^n)$ and modifying the multiplication operation. Thus $DN(q,n)= (GF(q^n), +, \cdot) ^ { \phi} = \big ( GF(q^n), +, \circ \big )$. We will denote a Dickson nearfield arising from the Dickson pair $(q, n)$ as $DN(q,n)$. For more details regarding the construction of the new multiplication operation denoted by '$\circ$', we refer the reader to \cite{dickson1905finite,pilz2011near}.\\

\begin{exa} (\cite{pilz2011near})
\label{ex2}Consider the field ($GF(3^{2})$, $+$, $\cdot$) with
\[GF(3^{2}) := \{0,1,2,x,1+x,2+x,2x,1+2x,2+2x\},\]
where $x$ is a zero of $x^{2}+1 \in \Bbb{Z}_{3}[x]$ with the new multiplication defined as

$$
a \circ b := \left\{\begin{array}{cc}
              a \cdot b     & \mbox{ if $a$  is a square in ($GF(3^{2})$, $+$, $\cdot$)}\\
              a \cdot b^3 & \mbox{ otherwise }
              \end{array}
\right.
$$
This gives the smallest finite  Dickson nearfield $DN(3,2):=(GF(3^{2})$, $+$, $\circ)$, which is not a field. Here is the table of the new operation $ \circ$ for $DN(3,2)$.

\[
 \begin{array}{r|ccccccccc}
\circ     & 0 & 1             & 2            & x     & 1+x  & 2+x  & 2x   & 1+2x & 2+2x \\ \hline
0         & 0 & 0             & 0            & 0          & 0         & 0         & 0         & 0         & 0\\ 
1         & 0 & 1             & 2            & x     & 1+x  & 2+x  & 2x   & 1+2x & 2+2x \\    
2         & 0 & 2             & 1            & 2x    & 2+2x & 1+2x & x    & 2+x  & 1+x \\    
x    & 0 & x        & 2x      & 2          & 1+2x & 1+x  & 1         & 2+2x & 2+x \\    
1+x  & 0 & 1+x      & 2+2x    & 2+x   & 2         & 2x   & 1+2x & x    & 1 \\    
2+x  & 0 & 2+x      & 1+2x    & 2+2x  & x    & 2         & 1+x  & 1         & 2x \\    
2x   & 0 & 2x       & x       & 1          & 2+x  & 2+2x & 2         & 1+x  & 1+2x \\    
1+2x & 0 & 1+2x     & 2+x     & 1+x   & 2x   & 1         & 2+2x & 2         & x \\    
2+2x & 0 & 2+2x     & 1+x     & 1+2x  & 1         & x    & 2+x  & 2x   & 2     
 \end{array}
\]

We will refer to this example  in later sections.
\end{exa}
The concept of a ring module can be  extended to a more general concept called  a nearring module where the set of scalars is taken to be a nearring.
\begin{defn}
	An additive group $(M,+)$ is called (right) nearring module over a (left) nearring $R$ if there exists a mapping,
	\begin{align*}
		\eta: \thickspace & M \times R \to M \\
		& (m,r) \to mr
	\end{align*} such that $m(r_1+r_2)=mr_1+mr_2$ and $m(r_1r_2)= (mr_1)r_2$ for all $r_1,r_2 \in R$ and $m \in M.$
	
	We write $M_R$ to denote that $M$ is a  (right) nearring module over a (left) nearring  $R$.
\end{defn}
\begin{defn}(\cite{djagbahowell18})
	A subset $A$ of a nearring module $M_R$ is called a $R$-subgroup if  $A$ is a subgroup of $(M,+),$ and  $AR= \lbrace ar \vert a \in A, r \in R \rbrace \subseteq A. $
\end{defn}
\begin{defn} (\cite{djagbahowell18})
	A nearring module $M_R$ is said to be irreducible if $M_R$ contains no proper $R$-subgroups. In other words, the only $R$-subgroups of $M_R$ are $M_R$ and $\lbrace 0 \rbrace.$
\end{defn}
\begin{cor}(\cite{djagbahowell18})
	Let $M_R$ be a unitary $R$-module. Then $M_R$ is irreducible if and only if $mR=M_R$ for every non-zero element $m \in M.$
	\label{cor}
\end{cor}
\begin{defn} (\cite{djagbahowell18})Let $M_R$ be a nearring module.
	$ N  $ is a submodule  of $M_R$ if :
	\begin{itemize}
		\item $ (N,+)$ is normal subgroup of $(M,+),$
		\item $(m+n)r-mr \in N$ for all $m \in M, n \in N$ and $r \in R.$
	\end{itemize}
\end{defn}

\begin{pro}(\cite{djagbahowell18})
	Let $N$ be a submodule of $M_R.$ Then $N$ is a $R$-subgroup of $M_R.$
	\label{pro}
\end{pro}
Note that the converse of this proposition is not true in general. In his   thesis (\cite{djagbahowell18}, page $14$) Beidleman gives a counter example. However,
\begin{lem}
	If $M_R$  is a ring module,  then the notions of $R$-subgroup and submodule of $M_R$ coincide.
\end{lem}
\begin{proof}
	By Proposition \ref{pro}, every $R$-submodule is a $R$-subgroup. Let $H$ be a $R$-subgroup of $M_R.$ Then $hr \in H$ for all $h \in H$ and $r \in R.$ But $hr=(m+h)r-mr$ for all $m \in M.$ Hence $H$ is a submodule of $M_R.$
\end{proof}
\begin{thm}(\cite{djagbahowell18})
	Let $R$ be a nearring that contains a right identity element $e \neq 0.$ $R$ is division nearring if and only if $R$ contains no proper $R$-subgroups.
	\label{irre}
\end{thm}
\begin{rem}
	Let $R$ be a nearfield. By Theorem \ref{irre}, $R_R$  is irreducible $R$-module. Thus $R$ contains only $\{ 0 \}$ and $R$ as submodules of $R_R.$
\end{rem}

\begin{defn}(\cite{djagbahowell18})
	Let $\lbrace M_i \vert i \in I \rbrace $ be a collection of submodules  of the nearring module $M_R$.  $M_R$ is said to be a direct sum of the submodules $ M_i,$ for  $ i \in I, $ if  the additive group $(M,+)$ is a direct sum of the normal subgroups $ (M_i,+),$ for  $ i \in I $. In this case we write $M_R= \bigoplus _{i \in I} M_i.$
\end{defn}

\begin{pro}(\cite{djagbahowell18})
	$M_R= \sum _ {i \in I} M_i$ and every element of $M_R$ has a unique representation as a finite sum of elements chosen from the submodules $M_i$ if and only if $ M_R= \sum _{i \in I} M_i$ and $M_k \cap \sum _{i \in I, i \ne k} M_i= \lbrace 0 \rbrace.$
\end{pro}

We also have that
\begin{pro}(\cite{djagbahowell18})
	Let $\lbrace M_i \vert \thickspace  i \in I \rbrace $ be a collection of submodules  of the nearring module $M_R$. Then $M_R = \bigoplus _{i \in I} M_i$ implies that $M_R= \sum_{i \in I } M_i$ and the elements of any two distinct submodules permute.
	\label{rp}
\end{pro}
According to the definition of a nearring module, there is no distributivity of elements of $R$ over the elements of $M$. If we consider $M_R$ as direct sum of the  collection of submodules  $\lbrace M_i \vert \thickspace i \in I \rbrace$ of the nearring module $M_R$, the following result enables us to distribute the elements of $R$ over elements contained in distinct submodules within the direct sum. This result holds significant utility within the concept of Beidleman near-vector spaces.

\begin{lem}(\cite{djagbahowell18})Let $M_R = \bigoplus _{i \in I} M_i,$  $M_i$ is a submodule of $M_R.$ If $m =\sum_{i \in I} m_i$ where $m_i \in M_i$ and $r \in R$ then
	\begin{align*}
		mr= \big ( \sum_{i \in I} m_i \big ) r= \sum_{i \in I}( m_ir).
	\end{align*}
	\label{lemm}
\end{lem}

\begin{defn}(\cite{djagbahowell18})
	A nearring module $M_R$ is called strictly semi-simple if $M_R$ is a direct sum of irreducible submodules.
\end{defn}
We now have,
\begin{defn}(\cite{djagbahowell18}) Let $(M,+)$ be a group.
	$M_R$ is called Beidleman near-vector space if $M_R$ is a strictly semi-simple $R$-module  where $R$ is a nearfield.
\end{defn}

\begin{thm}[\cite{djagbahowell18,beidleman1966near}] Let $R$ be a (left) nearfield and $M_R$ a (right) nearring module. $M_R$ is
			a finite dimensional near-vector space if and only if $ M_R $ is isomorphic to $R^n$ for some positive integer 
			$n$.
		\end{thm}

\subsection{Subgroups of $R^n$}
		
	In \cite{djagbahowell18}, $R$-subgroups of finite-dimensional near vector spaces were classified using the Expanded Gaussian Elimination (EGE) algorithm. This algorithm constructs the smallest $R$-subgroup containing a given finite set of vectors. It's important to note that such an $R$-subgroup always exists since the intersection of subgroups is also a subgroup.

		\begin{defn}
			Let $V$ be a set of vectors. Define $gen(V)$ to be the intersection of all $R$-subgroups containing $V$.
		\end{defn}
			\label{sec:linearity-index}
		Let $LC_0(v_1,v_2,\ldots,v_k):=\{ v_1,v_2,...,v_k\}$ and for $n\geq0$, let $LC_{n+1}$ be the set of all linear combinations of elements in $LC_n(v_1,v_2,\ldots,v_k)$, i.e.
		\begin{equation*}
			LC_{n+1}(v_1,v_2,\ldots,v_k)=\left \{ \sum_{i=1}^\ell w_i\lambda_i
			\;|\; \ell\ge 0, w_i\in LC_n,\lambda_i\in R \;\forall 1\le i\le\ell \right\}.
		\end{equation*}
		
		\begin{thm}[Theorem 5.2 in~\cite{djagbahowell18}]
			\label{thm:gen-lc} Let $v_1,v_2,\ldots,v_k \in R^n$. We have
			\begin{equation*}
				gen(v_1,\ldots,v_k)=\bigcup_{i=0}^\infty LC_i(v_1,\ldots,v_k).
			\end{equation*}
		\end{thm}
		
		Let $M_R$ be a nearring module. Let $V\subseteq M_R$ and let $T$ be an $R$-subgroup of $M_R$.
		\begin{defn}[Seed set]
			We say that \emph{$V$ generates $T$} if $gen(V)=T$.
			In that case we say that $V$
			is a \emph{seed set} of $T$. 
			We also define the \emph{seed number} $seed(T)$
			to be the cardinality of a smallest seed set of $T$.
		\end{defn}
		
		In \cite{djagba}, it was demonstrated that each $R$-subgroup can be expressed as a direct sum of modules $u_iR$ of a special kind:
		
		\begin{thm}[Theorem~5.12 in~\cite{djagbahowell18}]
			\label{thm:ege}
			Let $R$ be a proper nearfield
			and $\{v_1,\ldots,v_k\}$ be vectors in $R^n$. Then,
			$gen(v_1,\ldots,v_k)=\bigoplus_{i=1}^\ell u_iR$,
			where the $u_i$ are rows of some matrix $U=(u_{ij})\in R^{\ell\times n}$
			such that each of its columns has at most one non-zero entry.
		\end{thm}
	The EGE algorithm, is presented below, illustrates the proof of the theorem mentioned above. It is employed to compute the smallest $R$-subgroup for a given set of vectors.

\begin{proof}
Given a particular set of vectors $v_1,\ldots,v_k$, arrange them in a matrix $V$ whose $i$-th row is composed of the components of $v_i$, i.e., $V=(v_i^j)$ where $1 \leq j \leq n$. Then $gen(v_1,\ldots,v_k)$ is the $R$-row space of $V$, which is a $R$-subgroup of $R^n$. We can then do the usual Gaussian  elimination on the rows. The $gen$ spanned by the rows will remain unchanged with each operation (swopping rows, scaling rows, adding multiples of a row to another). When the algorithm terminates, we obtain a matrix $W \in R^{k \times n }$ in reduced row-echelon form (denoted by $RREF(V)$). Let the non-zero rows of $W$ be denoted by $w_1,w_2,\ldots,w_t$ where $t \leq k$.
\begin{itemize}
\item[Case 1:] Suppose that every column has at most one non-zero entry, then
\begin{equation*}
	gen(v_1,\ldots,v_k)=gen(w_1,\ldots,w_t)=w_1R+w_2R+\cdots+w_tR
\end{equation*} 
where the sum is direct. In this case we are done. 
\item[Case 2:] Suppose that the $j$-th column is the first column that has two non-zero entries, say $w_r^j\neq 0 \neq w_s^j$ with $r<s$, (we necessarily have $r,s\leq j$) where $w_r^j$ is the $j$-th entry  of row $w_r$ and $w_s^j$  the $j$-th entry  of row $w_s.$
Let $ \alpha, \beta, \gamma \in R$ such that $(\alpha +\beta) \lambda \neq \alpha  \lambda  + \beta  \lambda .$
We apply what we will call the \textsl{distributivity trick}:

Let $\alpha'= (w_r^j)^{-1}\alpha$ and $\beta'= (w_s^j)^{-1}\beta$. Then consider the new row
\begin{equation*}
	\theta=(w_r \alpha' + w_s  \beta') \lambda - w_r  (\alpha' \lambda)-w_s (\beta ' \lambda). 
\end{equation*}
Since $\theta \in LC_2(w_r,w_s)$ we have $\theta\in gen(w_1,\ldots,w_t)$.

 For $ 1 \leq l < j,$ either $w_r^l$ or $w_s^l$ is zero because the $j$-th column is the first column that has two non-zero entries, thus $\theta^l=0$. Note that by the choice of $\alpha,\beta,\lambda$, we have 
\begin{align*}
\theta^j &=(w_r^j)\alpha ' + w_s^j \beta') \lambda - (w_r^j\alpha ') \lambda- (w_s^j\beta ') \lambda \\
&=\big ( w_r^j (w_r^j)^{-1} \alpha + w_s^j(w_s^j)^{-1}  \beta \big ) \lambda- \big (  w_r^j ( w_r^j )^{-1} \alpha \big ) \lambda - \big (  w_s^j ( w_s^j )^{-1} \beta \big ) \lambda \\
&=(\alpha +\beta) \lambda - \alpha \lambda - \beta \lambda \neq 0.
\end{align*}
It follows  that $\theta^j\neq 0$.
Hence $ \theta =(0, \ldots, 0,\theta^j,\theta^{j+1},  \ldots, \theta^n )$. We now multiply the row $\theta$ by $ (\theta^j)^{-1},$ obtaining the row $\phi=(0, \ldots, 0,1,\theta^{j+1}(\theta^j)^{-1},  \ldots, \theta^n (\theta^j)^{-1}) \in gen(w_1,\ldots,w_k)$ where  $\phi ^j=1$ is  the pivot that we have created.

 As a next step, we form a new matrix of size $(t+1) \times n $ by adding $\phi$ to the rows $w_1, \ldots,w_t.$ On  this augmented matrix we replace the rows $w_r, w_s$  with $ y_r=w_r-(w_r^j) \phi, y_s=w_r-(w_s^j) \phi$ respectively. This yields another new matrix  composed of  the rows $w_1,\ldots,w_{r-1},y_r, \ldots, y_s, \phi, w_{s+1},  \ldots,w_t $ which has only one non-zero entry in the $j$-th column. By Lemma \ref{l1}, the  \textit{gen} of the rows of the augmented matrix is the  \textit{gen} of the rows of $W$ (which in turn is  $gen(v_1,\ldots,v_k)$).
Hence
\begin{align*}
gen(v_1, \ldots,v_k) & = gen(w_1,\ldots,w_r, \ldots, w_s, \ldots,w_t)\\
&= gen(w_1,\ldots,y_r, \ldots, y_s, \phi, \ldots,w_t).
\end{align*}

Continuing this process, we can eliminate all columns with more than one non-zero entry. Let the final matrix have rows $u_1,u_2,\ldots, u_{k'}$. Then
\begin{equation*}
	gen(v_1,\ldots,v_k)  =gen(w_1,\ldots,w_t)=gen(u_1,\ldots,u_{k'})=u_1R+u_2R+\ldots+u_{k'}R, 
\end{equation*} 
where the sum is direct.
\end{itemize}
\end{proof}

\newpage

\section{Representation of linear maps}
In vector spaces, matrix representation is motivated by the fact that, for a given basis, a linear mapping is well-defined when we specify the images of all elements in the basis. In a vector space $M_R$ with a basis $X={x_i,i\in I}$, we can define a linear mapping $T$ by specifying the images of each element of $X$. This can be expressed as $\left(\sum_{i\in I}x_ir_i\right)T= \left(\sum_{i\in I}x_iT\right)r_i$. Any matrix can serve as a representation of such a linear mapping, where each column of the matrix corresponds to the image of an element in the basis.

In a vector space, a linear mapping is uniquely determined when we specify its behavior on a basis. However, in a near-vector space, setting the image of the elements of a basis and following the same rules as in vector spaces may lead to a function that is not a linear mapping.

We consider the mapping $T$ from $R^2$ to itself, where $R$ is the Dickson Nearfield $DN(3,2)$. It is defined as follows: $(0,1)T=(1,1)$, $(1,0)T=(1,2)$, and for all $(a,b)$ in $R^2$, $(a,b)T=(0,1)Tb+(1,0)Ta$. Under this mapping, we find that $(1,X)TX=(X+2,X+1)$  but $(1,X)XT=(2X+1,2X+2)$ for $X$ satisfying $X^2+1=0$.

\subsection{Case for $DN(3,2)^2$}

In this section, we will examine the case of $DN(3,2)$ and its finite-dimensional near-vector spaces.

We will start by considering $R^2$, and let $\mathcal{M}$ be the set of all mappings from $R^2$ to itself. These mappings are defined by specifying the images of $(1,0)$ and $(0,1)$, and for all $(x_1,x_2)$ in $R^2$, $(x_1,x_2)T= (1,0)Tx_1+(0,1)Tx_2$. It's worth noting that $|\mathcal{M}|= 6561$. We see that. 

\begin{align*}
(x_1,x_2)T +(x_1',x_2')T&=  (1,0)Tx_1+(0,1)Tx_2 +(1,0)Tx_1'+(0,1)Tx_2'\\
&=   (1,0)Tx_1+(1,0)Tx_1' +(0,1)Tx_2 +(0,1)Tx_2'\\
&=(1,0)T(x_1+x_1')+(0,1)T(x_2+x_2') \text{ (left distributivity)},
\end{align*} 

but
\begin{align*}
((x_1,x_2)T )r= ( (1,0)Tx_1 + (0,1) Tx_2)r\\
((x_1,x_2) r)T  = (x_1r,x_2r)T=  (1,0)Tx_1r + (0,1) Tx_2r
\end{align*}

which can be different because of the lack of right distributivity. We want to explore the characterizations of linear mappings.

\begin{pro} \label{family}
 If $T$ in $\M$ is  defined by $(1,0)T=a$ and $(0,1)T=b$ is a linear mapping then all mappings $T'$ of $\M$ defined by $(0,1)T'=ar$ and $(0,1)T'=br'$ is a linear mapping also.
\end{pro}

\begin{proof}
\begin{align*}
(x_1,x_2)&=(1,0)x_1+(0,1)x_2\\
(x_1,x_2)T'&=arx_1+br'x_2=(rx_1,r'x_2)T\\
((x_1,x_2)+(x_1',x_2'))T'&= (x_1+x_1',x_2+x_2')T'\\
&=(r(x_1+x_1'),r'(x_2+x_2') )T\\
&=  (rx_1,r'x_2)T+(rx_1',r'x_2')T\\
&= (x_1,x_2)T'+(x_1',x_2')T'\\
((x_1,x_2)T')\lambda &=(rx_1,r'x_2)T \lambda\\
&= (rx_1\lambda,r'x_2\lambda)T \\
&=(x_1\lambda,x_2\lambda)T'\\
&=  (x_1 ,x_2)\lambda T'
\end{align*}

\end{proof}

We will define the representation matrix of a linear mapping as it is used in vector space, and we see that for $R^2$, the representation matrix of a linear mapping is has at most one non-zero element in each row.

Next, we determine the normal linear mappings, and we find that their representation matrices have at most one non-zero element in each row and each column. There are a total of 161 normal linear mappings. We find that the number of linear mappings in $R^2$ is 289.

    \subsection{Generalization}

Let $R$ be a nearfield and $n$ an integer. We will consider the near vector space  $R^n$ with a basis $X=\{x_1,\ldots, x_n\}$. 

Let $\cal M$ be the set of all mappings of $R^n$ to itself such that $x_iT=a_i$ for $i=1,\ldots, n$ where $a_i$ is an element in $R$ . For any element $m$ of $R^n$ such that $m=\sum_{i=1}^n x_ir_i$, $r_i$ in $R$, we have $mT= \sum_{i=1}^n a_ir_i$. 

Such mapping can be represented as a matrix $n\times n$ over the nearfield $R$ by setting as columns the $a_i$ as it is in vector space. Let us first prove that the element of $\cal M$ are homomorphisms of the additive group $(R^n,+)$. 

\begin{lem}
Let $T$ be an element of $\cal M$, then $T$ is an homomorphism of the additive group $(R^n,+)$. 
\end{lem}
\begin{proof}

Let $m_1,m_2$ be in $R^n$,

\begin{align*}
m_1T+m_2T&= \sum_{i=1}^n a_{i1}r_{i1}+\sum_{i=1}^n a_{i2}r_{i2}\\
&=\sum_{i=1}^n (a_{i1}r_{i1}+a_{i2}r_{i2})&\text{(by the commutativity of the additive group)}\\
&=\sum_{i=1}^n a_{i1}(r_{i1}+r_{i2})&\text{ (by the left distributivity)}\\
&=(m_1+m_2)T
\end{align*}
\end{proof}

Now we want to see which of those mapping are linear. Then, let us denote $M$ to be   the matrix representation of $T$.
 
 \begin{thm} \label{characoflinear}
Let $T $ be an element of $M$, then  b $T$ is linear if and only if the matrix representation associated with $T$ has at most one non-zero element in each row.
\end{thm}

\begin{proof}
Let $T$ be in $\M$, $m=\sum_{i=1}^{n}x_ir_i$ in $R^n$ and $r$ in $R$.
Let $M = \begin{pmatrix} M_{11}& \dots& M_{1n}\\
\vdots& \ddots&\vdots\\
M_{n1}&\ldots & M_{nn}\end{pmatrix}$.

Then,

$mTr= \begin{pmatrix}(\sum_{j=1}^{n} M_{1j}r_j)r \\
\vdots \\
(\sum_{j=1}^{n} M_{nj}r_j)r  \end{pmatrix}$ and $mrT= \begin{pmatrix}\sum_{j=1}^{n} T_{1j}r_jr \\
\vdots \\
\sum_{j=1}^{n} T_{nj}r_jr  \end{pmatrix}$

So by definition of linear mapping, $T$ is linear if and only if,  $mMr=mrM$ which means  $\sum_{j=1}^{n} T_{ij}r_jr= (\sum_{j=1}^{n} T_{ij}r_j)r $ for all $i$ in $\{1,\ldots,n\}$.

If the  $M$ has only one non-zero element, then it is obvious that the equality is true.

Conversely, let us suppose that $T$ is a linear mapping, i.e $\sum_{j=1}^{n} T_{ij}r_jr= (\sum_{j=1}^{n} T_{ij}r_j)r $ for all $i$ in $\{1,\ldots,n\}$ and there exist one row containing more than one element.

Let $i$ be such row and let $s,t$ the be the   first two columns such that the component is not $0$.  We know that because of the lack of right distributivity, we have elements $\alpha, \beta$ and $\gamma$ in $R$ satisfying $(\alpha+\beta)\gamma\neq \alpha\gamma+\beta\gamma$.

Let $m$ be the element of $R^n$ which component are $0$ except at the $s^{th}$ and $t^{th}$ the  component which are are respectively $r_s=M_{is}^{-1} \alpha$ and $r_t=M_{it}^{-1} \beta$.

Then the $i^{th}$ component of  $(mT)\gamma$   is  
 
\begin{align*}
((mT)\gamma)_i&=\left( \sum_{j=1}^{n} M_{ij}r_j\right)\gamma\\
&=(\alpha+\beta)\gamma
\end{align*} since $M_{ij}r_j=0$ for all $j\notin \{s,t\}$ by definition of $m$ and we have \begin{align*}M_{is}r_s=  M_{is} M_{is}^{-1} \alpha=\alpha\\M_{it}r_t=  M_{it} M_{it}^{-1} \beta=\beta.
\end{align*}

and the $i^{th}$ component of  $(m\gamma T)$   is  \begin{align*}
((m\gamma T))_i&=\left( \sum_{j=1}^{n} M_{ij}r_j\gamma \right)\\
&=(\alpha\gamma+\beta\gamma)
\end{align*}
 since $M_{ij}r_j\gamma=0$ for all $j\notin \{s,t\}$ by definition of $m$ and we have \begin{align*}M_{is}r_s\gamma=  M_{is} M_{is}^{-1} \alpha\gamma=\alpha\gamma\\M_{it}r_t\gamma=  M_{it} M_{it}^{-1} \beta\gamma=\beta\gamma.
\end{align*}
By the choice of $\alpha, \beta$ and $\gamma$, we have $m\gamma T \neq (mT)\gamma$ which is a contradiction. So all rows have at most one non-zero component.
\end{proof}

\begin{thm}
Let $T$ be a linear mapping. Then $T$ is normal if and only if,  the matrix representation $M$ of $T$ 
has at most one non-zero element in each row and   column.
\end{thm}
\begin{proof}
Let $T$ be a linear mapping with matrix representation $M = \begin{pmatrix} M_{11}& \dots& M_{1n}\\
\vdots& \ddots&\vdots\\
M_{n1}&\ldots & M_{nn}\end{pmatrix}$.

We know from (\cite{djagbahowell18} Theorem  6.1) that all subspace of $R^n$ are in form $e_1R\times \ldots u_kR$ where $e_i$ have only one non-zero component which is $1$.

Then $R^nT= \sum_{i=1}^n a_iR$ and $a_iR\cap  \sum_{j\neq i}  a_jR=\{0\}$ since a row contains at most one non-zero element.

Suppose that all $a_i$ has at most one non-zero component. Then the image of $T$ is $\bigoplus a_iR = \underset{a_k\neq (0,0,\ldots,0)} {\bigoplus}e_kR$ where $e_k$ is obtained by multiplying $a_k$ with the inverse of its non-zero component. So  $R^nT$ is  clearly a subspace.  Thus $T$ is normal. 

Conversely, let us suppose that $R^nT= \sum_{i=1}^n a_iR$ is a subspace and suppose  that the matrix representation $M$ of $T$ has a column $j$ with more than one non zero element say at position $j_1,\ldots, j_k$. 

We know that $a_jR$ is an $R$-subgroup of $R^n$
 which is contained in the subspace $A= e_{j_1}R+\ldots +e_{j_k}R$. 
 where $e_{j_i}$ is a vector with only one non-zero
  components $1$   in position $j_i$ . 
  Then as $A$ is a subspace, 
  it is a near vector space and $a_jR$ is 
  a proper  $R$-subgroup of $A$. 
  
Hence, there exist a non zero element $m$ in 
$A$, $a$ in $a_jR$ and $r$ in $R$ such that 
$(m+a)r-mr$ is not in $a_jR$.

But this element is not as well in $\sum_{i\neq j}  a_iR$  
since $a_i$ has $0$ as component in row $j_l$ for each $l$. 

Thus, there is contradiction and the column have at most one non-zero component.
\end{proof}

The sets of normal linear mappings and linear mappings are not nearrings but are closed under multiplication. In contrast to the theory of vector spaces and Andre near vector spaces, we make the following observations.
 Define

 $L(R^n)=\{ T:R^n\to R^n|T \text{ is normal linear mapping }\}$,  \\
  $Hom(R^n)=\{ T:R^n\to R^n|T \text{ is linear mapping }\}$ and \\
$Iso(R^n)=\{ T:R^n\to R^n|T \text{ is normal linear mapping  and bijective}\}$

\begin{pro} We have 

\begin{itemize}
	\item[(1)] $L(R^n,R^n)$ and $Hom(R^n)$ are not nearrings and does not form Beidleman near vector spaces
	\item[(2)]   $L(R^n)$ and $Hom(R^n)$ are   closed under multiplication.
\end{itemize}
\end{pro}
\begin{proof}
	
	\begin{itemize}
		\item[(1)] We take our example of $DN(3,2)$ and the normal linear mapping $T$ and $T':\quad R^2  \to R^2$  defined by their respective matrices $\begin{pmatrix}
			0&0\\0&1
		\end{pmatrix}$ and $\begin{pmatrix}
			0&0\\1&0
		\end{pmatrix}$, then the matrix of $T+T'$ will be  $\begin{pmatrix}
			0&0\\1&1
		\end{pmatrix}$. So $T+T'$ is not a linear mapping by our characterization in Proposition \ref{characoflinear}. So $L(R^n)$ and $Hom(R^n)$ are not closed under addition. Hence they do not form group under addiction.
		\item[(2)]   Let $M$ and $M'$ be two representative matrices of linear mappings $T$ and $T'$.
		
		We have $(MM')_{ij}= \sum_{k=1}^n M_{ik}M'_{kj} $, if  the row $i$ of $M$ has only zero elements, then $(MM')_{ij}=0$ for all $j$.  If   the row $i$ of $M$ has one non zero element in column $k$, then $(MM')_{ij}$ will be non zero if the element of  row $k$ of the column $j$  is also non zero but such situation occurs at most once since the row $k$ of $M' $ has at most one non zero element. 
		
		So when we do the multiplication $MM'$, we will always have at  most one non-zero element in each row.  Hence $T\circ T'$ is a linear mapping.
		
		Then, suppose  they are both normal, if  the column $j$ has only zero elements, then $(MM')_{ij}=0$ for all $i$.   if at most one row has non zero element, then the product will be obviously normal. Let us suppose there are two rows $i_1,i_2$ with non zero element at columns $k_1,k_2$.
		So we have $k_1\neq k_2$. Else if   the column $j$ of $T'$ has one non zero element in column $k$, then $MM'_{ij}$ will be non zero if the element of  column $k$ of the row $i$ of $M$  is also non zero but such situation occurs at most once since the column $k$ of $M$ has at most one non zero element. 
	\end{itemize}

\end{proof}

\section{Construction of a seed set }

In contrast to the concept of vector spaces, in which there do not exist $k$ vectors that span the entire finite-dimensional vector space $F^m$ when $k<m$ and $F$ is a field, the case of finite-dimensional Beidleman near-vector spaces exhibits a different behavior. Here, there exist vectors that, when combined, span the entire space $R^m$ where $R$ is a proper nearfield. This phenomenon is exemplified in Theorem 5.12 of \cite{djagbahowell18}, where the authors classified the $R$-subgroups of $R^m$ generated from a finite set of vectors.

During the process of explicitly describing the smallest $R$-subgroup containing a given set of vectors, Theorem 5.12 in \cite{djagbahowell18} demonstrated that the union of $p$-linear combinations of these finite sets of vectors is utilized. If there exists a finite set of vectors in $R^m$ such that the smallest $R$-subgroup containing these vectors generates the entire space $R^m$, then there exists a minimum positive integer $p$ for which the $p$-linear combinations of these vectors yield the entire space $R^m$. An intriguing open question pertains to finding tight bounds on positive integers $p$ for which $p$-linear combinations of a finite set of vectors yield the entire space and investigating potential constructions of seed sets that yield finite-dimensional near-vector spaces. We now introduce the following concepts.

\begin{defn}
     A vector $u$ is left multiple of $v$ if there exists $r \in R$ such that $u=rv.$
 \end{defn}
 \begin{defn}
     Let $V \in R^{k \times m}$ be a matrix of $k$ rows and $m$ columns for
     $m\ge 2$. We will say that $V$ is $1$-column independent if for all $ 1 \leq i < j \leq m$, $\alpha \in R$, $v_i \neq \alpha v_j$.
 \end{defn}

\begin{defn}Let $R$ be a finite nearfield.
A finite set of vectors $V= \{v_1,\ldots, v_k \}$ in $R^m$ is called $\gamma$-linearly dependent for some positive integer $\gamma$ if there exists $v_i \in V$ such that $v_i \in LC_{\gamma}(v_1,\ldots,v_{i-1},\widehat{v_i}, v_{i+1},\ldots, v_k)$. We define $V$ to be $\gamma$-linearly independent if $V$ is not  $\gamma$-linearly dependent.
\end{defn}

\begin{defn} Let $m \geq 3$.
Let $R$ be a finite nearfield and $v_1, \ldots, v_k \in R^m$ be a finite set of vectors  such that $k \geq 2$. The set $LC_p(v_1, \ldots, v_k )$  will be called the  $p$-linear combinations of the vectors $v_1, \ldots, v_k $. We define the index of $R$-linearity of $v_1, \ldots, v_k \in R^m$  to be
\begin{align*}
I(v_1, \ldots, v_k )= \min \{ p \in  \mathbb{N} \thickspace : \thickspace LC_p(v_1, \ldots, v_k )= R^m\}
\end{align*}
 the smallest positive integer for which the $p$-linear combination of the vectors 
 
 $v_1, \ldots, v_k $ yields the whole space $R^m$.
 \label{df3}
\end{defn}

Suppose that there exists $V=\{v_1,\ldots, v_k \}$  a finite set of vectors in $R^m$  such that $gen(V)=R^m$. Since $\mathbb{N}$ is an ordered set then $I(V)$ is well-defined.   
\begin{exa}
Taking $m=3$, it has been shown by  Theorem~5.12 in~\cite{djagbahowell18}  that there exists $v_1=(1,0,1)$ and $v_2=(1,1,0)$ in $R^3$ such that $gen(v_1,v_2)=R^3.$ Note that $LC_2(v_1,v_2)=R^3$ and  $LC_1(v_1,v_2) \neq R^3$. Hence $I(v_1,v_2)=2.$
\end{exa}

We have the following interesting observation.

 \begin{thm}
       Let $V^{(t)}$ be the matrix after $t$ steps of the $EGE$ algorithm on the columns vectors 

       $\{ (a_1,\ldots, a_l), \ldots, (b_l, \ldots, b_l)\}$. We have the following:

       \begin{enumerate}
     \item  If $V^{(t)}$ is $1$-column dependant then $V^{(t+1)}$ is also  $1$-column dependant.
     \item      If $V^{t}$ is $1$-column independent then $V^{t+1}$ is also $1$-column independent.
       \end{enumerate}

 \end{thm}
 \begin{proof}~We prove the two statements.

 \textbf{1.}   We want to show that if $V^{(t)}$ is $1$-column dependent then $V^{(t+1)}$ is also $1$-column dependent. Consider, Without loss of generality   after $t$ steps of $EGE$,  the following vectors columns $ (v_1^1, \ldots, v_k^1)$ and  $(v_1^2, \ldots, v_k^2)$.   In the $EGE$ process, we take the linear combinations of the rows to form new rows. For example, let $\alpha_1,\ldots, \alpha_k, \lambda \in R.$ Consider $z_1=(\sum_{i=1}^k v_i \alpha_i) \lambda$. Then $z_1^1=(\sum_{i=1}^k v_i^1 \alpha_i) \lambda$. Assume that $z_1^2=sz_1^1$ for some $s \in R$ i.e, the columns $(z_1^1, 0,\ldots, 0)$ and $(z_1^n, 0,\ldots, 0)$ are not $1$-column independent. Then
 \begin{align*}
 (\sum_{i=1}^k v_i^2 \alpha_i) \lambda =s(\sum_{i=1}^k v_i^1 \alpha_i) \lambda
 =(\sum_{i=1}^ks v_i^1 \alpha_i) \lambda
 \end{align*}

 It follows that $v_i^2=sv^1_i$ for $i=1,\ldots,k$. Thus $ (v_1^1, \ldots, v_k^1)$ and  $(v_1^2, \ldots, v_k^2)$ are  $1$-column dependent.

   \textbf{2.} Let assume that there is no $s 
 \in R$ such that $a_i=s b_i$ for all $i.$ In the process of $EGE$ we take the linear combination of the rows to form the new rows. Our additional row is of the form $(\sum_{i=1}^la_i \alpha_i) \lambda,  (\sum_{i=1}^lb_i\alpha_i) \lambda$.  The update matrix of size $(l+1) \times 2$  will be constituted of the columns

  $( a_1-(\sum_{i=1}^la_i \alpha_i) \lambda, a_2, \ldots, a_l, (\sum_{i=1}^la_i \alpha_i) \lambda) $ and $(b_1-(\sum_{i=1}^lb_i \alpha_i) \lambda, b_2, \ldots, b_l, (\sum_{i=1}^lb_i \alpha_i) \lambda)  $. Let's assume that there exists $s$ such that for $1 \leq i \leq l +1$ we have $u_i=sv_i$.  For $i=l+1$ we have, $a_{l+1}=sb_{l+1}$ which implies that $(\sum_{i=1}^la_i \alpha_i) \lambda=s (\sum_{i=1}^lb_i \alpha_i) \lambda$. Hence $\sum_{i=1}^la_i \alpha_i=s (\sum_{i=1}^lb_i \alpha_i)$. Furthermore,
  $$a_1-(a_1-\sum_{i=1}^la_i \alpha_i) \lambda=s(b_1-(\sum_{i=1}^lb_i \alpha_i) \lambda)=sb_1-s(\sum_{i=1}^lb_i \alpha_i)=sb_1- (\sum_{i=1}^la_i \alpha_i)\lambda.$$ It follows that $a_1=sb_1$ which leads to contradiction.

\end{proof}
We also have.
\begin{lem} Let $R$ be a finite nearfield and $V= \{v_1,\ldots, v_k \}$ be a finite set of vectors in $R^m$ and $|R|=t$.  Then $\vert LC_1(V) \vert \leq t ^k$. Furthermore,  if $V$ is  $2$-linearly  independent every element of $LC_1(V)$ is unique, $\vert LC_1(V) \vert = t ^k$ and $k \leq m$.
\label{l3}
\end{lem}
\begin{proof}
Let $u,v \in LC_1(V)$ such that $u= \sum_{i=1}^kv_i \alpha_i$ and $v= \sum_{j=1}^kv_i \beta_j $ where $(\alpha_1, \ldots, \alpha_k) \neq (\beta_1, \ldots, \beta_k)$. Without loss of generality we can assume that  $\alpha_1 \neq \beta_1.$ Suppose that $u=v$. Then $v_1 \alpha_1-v_1 \beta_1=\sum_{i=2}^kv_i \alpha_i - \sum_{j=2}^kv_i \beta_j.$ It follows that
\begin{align*}
v_1= \big (  \sum_{j=2}^kv_i (\beta_i -\alpha_i \big ) \big )( \alpha_1 - \beta_1)^{-1}.
\end{align*}
 Thus $v_1 \in LC_2(v_2,\ldots,v_k)$. So $V$ is $2$-linearly dependent. We reach to contradiction. Therefore  if $V$ is $2$-linearly  independent then $\vert LC_1(V) \vert =t^k.$ Suppose that  $V$ is $2$-linearly  independent and  $k>m$, we  have  $ \vert LC_1(V) \vert = t^k> t^m$ and all the $t^k$ are distinct. It contradicts the fact that we have at most $t^m$ vectors in the space. Hence $t \leq m.$
\end{proof}

In analogy to the notion of a basis of a subspace in the theory of vector spaces,   we also have $R$-basis and $R$-dimension of an $R$-subgroup of the finite dimensional Beidleman near-vector spaces $R^m$. In the following we count the  $R$-subgroups of $R$-dimension $k$ of $R^m$.

\begin{defn}
  The number of $R$-subgroups  of $R^m$ of dimension $k$ of $R^m$ up to reordering of coordinate is the number of matrix obtained after $EGE$  without reordering the column.    
\end{defn}

 \begin{pro} 
  Let $R$ be a finite nearfield.
  The number of  $R$-subgroups of $R$-dimension $k$ of $R^m$ up to the reordering of coordinates is $$
  \sum _{t=k}^mp_{k}(t) (|R|-1)^{t-k},
 $$ where $p_k(t)$ is the number of partitions of $t$ into $k$ parts and $t$ is the total number of non-zero entries in  all the rows. 
  \label{th6}
 \end{pro}

 \begin{proof}
     Let $t=\sum_{i=1}^k \# u_i$ where $\# u_i$ is the number of non-zero entries in the row $u_i$. It is clear that $t \leq n.$ For $\# u_i=1$ for all $i=1, \ldots, k$ then $t=k$. Hence $k\leq t \leq n.$ Note that $t$ is the total number of non-zero entries in all columns or in all the rows. 

     Given $t$, we can partition $t$ into $k$ parts where each part (containing some non-zeros entry) will represent each row vector. We have $(|R|-1)^t$ possible choice of non-zero elements from $R^*$ to fill in the $t$ places of each partition.

     Since the $gen$ is unchanged for any permutations of row vectors and is unchanged for any scalar multiplies to others rows, then, for a given $t$, $N=\frac{1}{(|R|-1)^k}p_k(t)(|R|-1)^t$ is the number of $R$-subgroups of dimension $k$ up to reordering of coordinate. For  $k\leq t \leq m$ yield to  $\sum _{t=k}^mp_{k}(t) (|R|-1)^{t-k}$ is the number of  $R$-subgroups of $R$-dimension $k$ of $R^m$ up to the reordering
 \end{proof}

The form of the seed set of $DN(3,2)^m$ for $2 \leq m \leq 9$, as described in \cite{djagbahowell18}, motivates us to seek a possible general construction of seed sets for $R^m$, where $R$ is a nearfield and $m$ is a natural number. Consequently, we have developed Algorithm \ref{firstalgo}  (see in the appendix) to provide such a construction.

We will use the notation $V_m$ to represent the matrix output of Algorithm \ref{firstalgo} for a given input value of $m$, and $S_m$ to denote the ordered set of its row vectors, with ordering determined by their respective row numbers in $V_m$.

First, we found the following:

\begin{thm}
	
	Let $k\geq 1$ be an integer. Then, the maximal value of integer  $m$ such that $V_m$ has $k$-rows is given by the sequence $(u_k)_{k\geq 1}$ defined by: 
	
	\begin{align*}
		u_1&=1\\
		u_{k+1}&=u_k+(|R|-2) k+1
	\end{align*}
	
	The range of values of integer  $m$ such that $V_m$ has $k$-rows is given by:
	
	$\left[u_{k-1}+1,u_k\right]$
	
 \label{rrR}

\end{thm}

\begin{proof}
	Let $u_k$ be the maximal   value of integer  $m$ such that $V_m$ has $k$-rows.  
	
	We see that the construction of $V_m$ is obtained by adding eventual row/column to $V_{m-1}$. Also, we only add additional column when the maximum number of column that can be created with a given number rows is complete.  Thus, if we find $V_{m}$ has number of rows greater than $k$, then $m>k$. And if we find  $V_{m}$ has number of rows less than $k$, then $m<k$.
	
	We have $u_1=1$ because, for $m=1$, Algorithm \ref{firstalgo} return $(1)$, and for $m=2$, as we enter already in the while loop, we add a  new row so $1\leq u_1< 2$. It implies that $u_1=1$ as $u_1$ is an integer.

	Let us now  prove 
	that  $u_{k+1} = u_k+(|R|-2) k + 1$. From the construction, to construct  a matrix of $k+1$ rows, we need to complete the maximal number of columns that can have a matrix of $k$ rows.

	So first, we replace the identity matrix of range $k$ by a new identity matrix of range $k+1$ and complete the row with $0$. Which mean that the minimum number of column for $k+1$ is $u_k+1$. Then, we  will be able to add new column to this new construction as long as the $counter$ does not reach $k-1$ which mean the $counter$ can take $k$ values.  Then, for a given $counter$, we can add $|R|-2$ columns as we remove $0$ and $1$ to $R$.  So, we have $u_{k+1}=u_k+1+(|R|-2)k$.

\end{proof}

Here we give an explicit form of $u_k$ in function of $k$.

\begin{align*}
	u(0)&=0\\
	u(1)&=u(0)+(|R|-2)*0+1\\
	u(2)&=u(1)+(|R|-2)*1+1\\
	\vdots&\vdots
	\\u(k )&=u(k-1)+(|R|-2)*(k-1)+1\\
	u(k )&= (|R|-2)*\sum_{i=0}^{k-1}+k\\
	u(k )&= (|R|-2) \frac{(k-1)*k}{2}+k\\
	u(k )&=  \frac{(|R|-2)(k-1)*k+2 k }{2} \\
	u(k )&=  \frac{\left((|R|-2)(k-1) +2   \right)k }{2} \\
\end{align*}

Solving this second degree equality in regard to $k$ and taking the only one positive value, we have:

\begin{lem}\label{explicitk} From Theorem \ref{rrR}, the explicit expression of $k$ is given by
	$$k= \frac{\mathit{|R|} + \sqrt{\mathit{|R|}^{2} + 8 \, {\left(\mathit{|R|} - 2\right)} u_k - 8 \, \mathit{|R|} + 16} - 4}{2 \, {\left(\mathit{|R|} - 2\right)}}$$.

\end{lem}

Then we notice that the function  $f: m\mapsto  \frac{\mathit{|R|} + \sqrt{\mathit{|R|}^{2} + 8 \, {\left(\mathit{|R|} - 2\right)} m - 8 \, \mathit{|R|} + 16} - 4}{2 \, {\left(\mathit{|R|} - 2\right)}} $ is strictly increasing  (because  $\mathit{|R|}^{2} + 8 \, {\left(\mathit{|R|} - 2\right)} m - 8 \, \mathit{|R|} + 16$ is an affine function of $m$ of positive leading coefficient, square root is an increasing function and addition of number followed by multiplication with positive number does not change the  variation of a function) such that $f(u_k)= k $ for any integer $k\geq 1$. 

As we have $f(u_{k-1}) = k-1$ and $f(u_k)=k$, for any element $m$ in $[u_{k-1}+1; k]$ we have $k-1 <f(m)\leq k  $. It means that  any element $m$ is in $[u_{k-1}+1; k]$, $k = ceil(f(m))$. So we have:

\begin{lem}
	
	Let $m$ be a positive integer, then the number of rows of $V_m$ is given by:

	$k=\left\lceil   \frac{\mathit{|R|} + \sqrt{\mathit{|R|}^{2} + 8 \, {\left(\mathit{|R|} - 2\right)} m - 8 \, \mathit{|R|} + 16} - 4}{2 \, {\left(\mathit{|R|} - 2\right)}}\right\rceil $

\end{lem}

Now, we are ready to give the Theorem:

\begin{thm}
	For any $V_m$ obtained by the Algorithm \ref{firstalgo}, we have that the set of its row vectors $S_m$ is $R$-linearly independent and  $gen(S_m)=R^m$ (i.e., $S_m$ is a seed set of $R^m$). 

 \label{rr}
\end{thm}

\begin{proof}
Note that $S_m$ is $R$-linearly independent since by the construction of $V_m$ it is in reduced row echelon form.  We will do the proof of $gen (S_m)=R^m$ by induction in $m$ For $m=1$, we know that $\{1\}$ generates $R^m$. Let us assume that $gen(S_m)=R^m$ for  $m\geq 1$.	
	We have 2 cases:
 
	\textit{Case 1:}   $m=u(k)$ for a certain value of $k$ and $m+1=u(k)+1$. From the algorithm described in the Appendix,  the first $k$ rows of $V_{m+1}$ is obtained by inserting one $0$ in the $(k+1)^{th}$ position of every rows of $V_m$ of the same position.  When we perform the EGE Algorithm on the the  matrix $V_{m+1}$, 
		the first $k $ columns and the  will behave bijectively as for $V_m$. The new rows will have the component $0$   in the $k+1^{th}$ column since the $k+1^{th}$ components of each above vector are all $0$ and the last row will not be involved in the distributivity trick (by Theorem~5.12 in~\cite{djagbahowell18}, we see that the first two non-zero elements of the column $j$ will all have number of row less than $j$). For the $k+1^{th}$ column, we don't need to add any additional row since it has only one zero component which is already the only one in its column. The algorithm will work on remaining columns  with the same behaviour as for $V_m$ for it sub-matrix from column $k$. It means that the $m $ columns will will generates $m-k$ additional rows (By our hypothesis $gen(S_m)$) and in the end of the EGE Algorithm  we have $k+1 +m-k= m+1  $ rows. Thus, $gen(S_{m+1})=R^{m+1}$.		
		
  \textit{Case 2:} $m$ and $m-1$ is in the same interval  $\left[u(k-1)+1, u(k)\right] $. From the algorithm described in the Appendix, each row vector of $V_{m+1}$ is obtained by  adding only one $1$ or an element $s$ of $S$ in its last column. It means that the column $1$ to $ m$ generate $m-k$ additional vector in the running of the EGE algorithm.  The last column generate only one new non-zero row  $(0, \ldots,1)$ by definition of the $AGE$ algorithm, and we will not delete any of the previous rows because the  other rows have already one non-zero element in one of the previous column (by our hypothesis, it generate already $R^m$). Thus, the number of rows we obtain is $k+m-k+1=m+1$.

	In both cases we have $gen(S_{m+1})=R^{m+1}$. Thus, for any positive integer $m$, $gen(S_m)=R^m$.

\end{proof}


\section{Conclusion}
In this paper, we provided representations of linear and normal linear maps between finite-dimensional Beidleman near-vector spaces, and we derived algorithms for constructing seed sets for such spaces.

Let $R$ be a finite nearfield of size $q$. A rough upper bound on the size of $LC_2(v_1,\ldots,v_k)$ is $q^{q^k}$. We propose the following open question:
		
		\begin{ques}
			Does there exist any example of a near-vector space where
			$I(v_1,\ldots,v_k)> 2$ for some $v_1,\ldots,v_k$?
		\end{ques}
		
		More generally, can we find an  explicit expression, or at least some nontrivial bounds for 
		$I(v_1, \ldots$ $, v_k )$? 

\paragraph{Acknowledgments}
		This work was supported by the Council-funded from the Office of Research Development funding at Nelson Mandela University.

\section{Appendix}
The following algorithm derive the explicit expression of the seed  number $k$ to generate a given $R$-subgroups. Note that  the  Algorithm \ref{algo2}  is a recursive version of Algorithm \ref{firstalgo}.

 \begin{algorithm}
\caption{Create a seed set $V_m$ of $R^m$ and give its number of row}
\label{firstalgo}                          
\begin{algorithmic} 
\REQUIRE $m\in \mathbb{N}$, $R$ is a finite nearfield with unity $1$, the indexing of matrices, vector start from $0$

\ENSURE $V_m$ and $k$ such that $gen(V_m)=R^m$ and $V-m$ has $k$ columns.  
    \IF{$ m= 1$}
        \RETURN  $(1)$
    \ELSE
        
        \STATE $K\leftarrow (1)$ 
        
        \STATE  $S\leftarrow R- \{0,1\}$
        \STATE $k\leftarrow 1$ \COMMENT{number of row}
        \STATE $numcol\leftarrow1$  \COMMENT{ the index of column to fill next time/number of column already filled}
        
        \WHILE{ $numcol<m$} 
            \STATE $k\leftarrow k+1$ \COMMENT{need to add one more row}
            
            \STATE  $Kprim \leftarrow   I_k(R)$ \COMMENT{Identity matrix of order $k$}

            \STATE $numcol\leftarrow column\_number(Kprim) $
            \IF{$numcol= m$}
               \STATE  show $k$
                \RETURN $Kprim$ 
            \ENDIF 
            
            \STATE $K \leftarrow$ sub-matrix of $k$ from column $k-1$ \COMMENT{ Matrices column of K}
            
            \IF{ $K\neq ()$}
                
               \STATE  $K  \leftarrow    \begin{matrix}K\\ \hline \\ 0\, \ldots \, 0 \end{matrix}   $  \COMMENT{add a row of zero in $K$}

                \STATE $K \leftarrow  Kprim|(K)$ \COMMENT{Merge $Kprim$ and $K$ by column}
                
            \ELSE
               \STATE  $K\leftarrow Kprim$
            \ENDIF

            \STATE $numcol\leftarrow  column\_number(K) $

            \IF{ $numcol= m$}
                \RETURN $K$
            \ENDIF
            \STATE $counter\leftarrow 0$
            \STATE $S1 \leftarrow  copy(S)$
            \WHILE{$counter<k-1$}
                \WHILE{ $numcol<m$ and $S1\neq \emptyset$}
                    
                    \STATE $newcolumn\leftarrow  [ ]$

                    \FOR{$i$ in  $[0;counter  ]$}
                        
                        \STATE  append  $[R[2]]$   to newcolumn \COMMENT{fill the first counter rows with 1}
                     \ENDFOR    
                        
                    \FOR{  $i$ in $[ counter+1 ,k -1]$}
                        \STATE new column.append([S1[0]]) \COMMENT{Complete the other rows with S[0]}
                    \ENDFOR

                    \STATE  remove $S1[0]$ from $S1$

                    \STATE $ newcolumn\leftarrow  matrix(newcolumn)$

                    \STATE $K\leftarrow  K.augment(newcolumn)$

                    \STATE $numcol\leftarrow  numcol+1$

                    \IF {$numcol= m$}
                        \STATE show $k$
                        \RETURN $K$
                    \ENDIF
                    \ENDWHILE
                \STATE $counter\leftarrow  counter+1$
                \STATE $S1\leftarrow  copy(S)$
              \ENDWHILE  
           \ENDWHILE   
        \ENDIF     
        \STATE show $k$
        \RETURN  $K$ 
 
\end{algorithmic}
\end{algorithm}

 \begin{algorithm}
\caption{Create a seed set $V_m$ of $R^m$}
 \label{algo2}                          
\begin{algorithmic} 
\REQUIRE $m\in \mathbb{N}$, $R$ is a finite nearfield with unity $1$

\ENSURE $V_m$ such that $gen(V_m)=R^m$

 \STATE $k \leftarrow \left\lceil   \frac{\mathit{|R|} + \sqrt{\mathit{|R|}^{2} + 8 \, {\left(\mathit{|R|} - 2\right)} m - 8 \, \mathit{|R|} + 16} - 4}{2 \, {\left(\mathit{|R|} - 2\right)}}\right\rceil $ \COMMENT{upper bound of the seed number of $R^m$}
\IF{  m=1}
\STATE  \RETURN (1) 

\ELSE

\STATE            $K\leftarrow  0_{k\times m}$  \COMMENT{ The array we will complete, a zero matrix of $k$ rows and $m$ columns. }
\STATE           $ prevm \leftarrow  \frac{( (|R|-2) (k-1-1)+2) (k-1)}{2}$  \COMMENT{The biggest value of $n$ which have a seed set of cardinality $k-1$.  }
\STATE          $V \leftarrow  V_{ prevm}$  


 \STATE \FOR{i in [1;k]}  
                  \STATE    $ K[i][i] \leftarrow 1$ \COMMENT{Copy the identity matrix in $K$  in the  first $k$ columns.}
              \ENDFOR

 \STATE    \FOR{j in [k; prevm]}

    				\FOR{i in  [1,k] }
       				 \STATE  
       				
       				        $ K[i][j+1]\leftarrow V[i][j]$ \COMMENT{Copy the submatrix of $V_{prevm}$ from the $k$th column   in $K$  in the same row number but the column number shifted by $1$. (just after the identity matrix) }
       				        
      				 \STATE 
      				       $  K[k][j+1]\leftarrow 0     $  \COMMENT{ Add $0$ element in the last rows from column $k+1$ to column $prevm$}
 	 		 \ENDFOR

    \ENDFOR

       \COMMENT{Now I will complete the first row with 1}
        \STATE     $S \leftarrow R\ \{0,1\}$  \COMMENT{The non zero element of $R$ different to $1$}
        \STATE   counter=1

       \STATE     $numcol \leftarrow prevm+2$ \COMMENT{The number of the column to be filled, note that we already fill the first $prevm+1$ columns.}
        \STATE   \IF{numcol=m+1}
         \STATE        \RETURN   $K$ \COMMENT{We return $K$ as matrix because we already completed the required number of columns.}
         		\ENDIF

      \STATE   \WHILE{$counter< k$}   
      
        
        \STATE        \FOR{$i$ in $[numcol,m]$}
       \STATE           $K[counter][i] \leftarrow 1 $ \COMMENT{Fill with 1 the row counter}
          			 	 \ENDFOR

          \STATE      \FOR{j in S}
       \STATE           \FOR{ i in  [counter ,k]}
        \STATE                $ K[i][numcol] \leftarrow j $ \COMMENT{ Put a copy of S}
 						\ENDFOR

       \STATE             $ numcol \leftarrow numcol+ 1 $ \COMMENT{Change the number of the column after setting all number in the column}


              \STATE       \IF{ $numcol= m+1$} 
               \STATE           \RETURN $K$  \COMMENT{We return $K$ as matrix because, we manage to complete every columns of our matrix $K$} 
              				\ENDIF
              		\ENDFOR
             \STATE    $counter\leftarrow counter+1$ 
        \ENDWHILE
\ENDIF

\RETURN $K$
 
\end{algorithmic}
\end{algorithm}

\newpage

\bibliography{main}
		\bibliographystyle{plain}
		
\end{document}